\documentclass[a4paper,11pt,final]{amsart}

\usepackage{amsmath}
\usepackage{mathtools}
\usepackage[all]{xy}
\usepackage[latin1]{inputenc}
\usepackage{varioref}
\usepackage{amsfonts}
\usepackage{amssymb}
\usepackage{bbm}
\usepackage{graphicx}
\usepackage{mathrsfs}
 \usepackage[hypertexnames=false,backref=page,pdftex,
	pdfpagemode=UseNone,
	breaklinks=true,
	extension=pdf,
	colorlinks=true,
	linkcolor=blue,
	citecolor=magenta,
	urlcolor=blue,
]{hyperref}

\newcommand{\sA}{{\mathcal A}}

\newcommand{\sH}{{\mathcal H}}

\newcommand{\sO}{{\mathcal O}}

\newcommand{\sX}{{\mathcal X}}
\newcommand{\sY}{{\mathcal Y}}

\newcommand{\C}{{\mathbb C}}

\renewcommand{\L}{{\mathbb L}}

\newcommand{\N}{{\mathbb N}}
\renewcommand{\P}{{\mathbb P}}

\newcommand{\R}{{\mathbb R}}
\renewcommand{\S}{{\mathbb S}}

\newcommand{\gothm}{{\mathfrak m}}
\newcommand{\gothn}{{\mathfrak n}}

\newcommand{\gothp}{{\mathfrak p}}

\newcommand{\goths}{{\mathfrak s}}

\newcommand{\gothP}{{\mathfrak P}}

\newcommand{\abs}[1]{{\left|#1\right|}}
\newcommand{\an}{{\operatorname{an}}}
\newcommand{\Ann}{\operatorname{Ann}}
\newcommand{\Art}{\operatorname{Art}}

\newcommand{\ba}[1]{\overline{#1}}
\newcommand{\id}{{\rm id}}

\newcommand{\coker}{\operatorname{coker}}

\newcommand{\Gr}{{\rm Gr}}

\newcommand{\Hom}{\operatorname{Hom}}

\newcommand{\img}{\operatorname{im}}
\newcommand{\into}{{\, \hookrightarrow\,}}
\newcommand{\isom}{\cong}

\newcommand{\lt}{{\rm{lt}}}

\newcommand{\Mor}{\operatorname{Mor}}
\newcommand{\nil}{\operatorname{nil}}

\newcommand{\op}{{\rm op}}

\newcommand{\Psh}{{\rm Psh}}

\newcommand{\red}{{\operatorname{red}}}
\newcommand{\reg}{{\operatorname{reg}}}

\newcommand{\rk}{{\rm rk}}

\newcommand{\Schm}{{\rm Sch }}
\newcommand{\Sch}[1]{{\rm Sch / #1}}
\newcommand{\Set}{{\rm Set}}

\newcommand{\Spec}{\operatorname{Spec}}

\newcommand{\supp}{\operatorname{supp}}

\renewcommand{\to}[1][]{\xrightarrow{\ #1\ }}

\newcommand{\tensor}{\otimes}
\newcommand{\tOm}{\widetilde{\Omega}}

\newcommand{\vphi}{\varphi}

\newcommand{\ul}[1]{{\underline{#1}}}

\newcommand{\wlfun}[3]{{\underline{#1_{#2/#3}}}}

\newtheoremstyle{citing}% name
  {}%      Space above, empty = `usual value'
  {}%      Space below
  {\itshape}% Body font
  {}%         Indent amount (empty = no indent, \parindent = para indent)
  {\bfseries}% Thm head font
  {\textbf{.}}%        Punctuation after thm head
  {.5em}%     Space after thm head: " " = normal interword space;
  {\thmnote{#3}}% Thm head spec

\theoremstyle{plain}

\newtheorem{Thm}[subsection]{Theorem}

\theoremstyle{definition}

\newtheorem{Cor}[subsection]{Corollary}
\newtheorem{Def}[subsection]{Definition}
\newtheorem{Ex}[subsection]{Example}%[section]

\newtheorem{Lem}[subsection]{Lemma}

\newtheorem{Prop}[subsection]{Proposition}

\numberwithin{equation}{section}

\theoremstyle{remark}

\newtheorem{Rem}[subsection]{Remark}
{\theoremstyle{citing}
\newtheorem*{custom}{}}

\newcommand{\wl}{{Z/W}}
\newcommand{\wls}{{\rm wl}}

\newcommand{\pq}{{p,q}}

\title[Normal crossing singularities]{Normal crossing singularities and Hodge theory over Artin rings}
\author{Christian Lehn}
\address{Christian Lehn\\Institut Fourier\\
100 rue des maths\\38402 St Martin d'Heres\\
France}
\email{christian.lehn@ujf-grenoble.fr}

\begin{document}
\thispagestyle{empty}

\begin{abstract}
We develop a Hodge theory for relative simple normal crossing varieties over an Artinian base scheme. We introduce the notion of a mixed Hodge structure over an Artin ring, which axiomatizes the structure that is found on the cohomology of such a variety. As an application we prove that the maps between the graded pieces of the Hodge bundles have constant rank.
\end{abstract}

% 13D10 Deformations and inﬁnitesimal methods
% 14C30 Transcendental methods, Hodge theory
% 32S35 Mixed Hodge theory of singular varieties
 

\maketitle

\setlength{\parindent}{0em}
\setcounter{tocdepth}{1}

\tableofcontents

\section*{Introduction}\label{sec intro}

In \cite{Fr83} the object of study were simple normal crossing varieties $Y$ with K\"ahler components. Although the article focused on deformations, there were a lot of general results proven regarding the Hodge theory of these objects. 

We consider locally trivial families of simple normal crossing varieties over an Artinian base and prove Hodge theoretic results in the style Friedman in this setting. We introduce the notions of a \emph{mixed Hodge structure} and a \emph{mixed Hodge-Weil structure}, both \emph{over a local Artin $\C$-algebra} $R$. We have

\begin{custom}[Theorem \ref{thm mixed}]
Let $Y$ be a proper, simple normal crossing $\C$-variety and let $f:\sY \to S$ be a locally trivial deformation of $Y$ over $S=\Spec R$ for an Artinian local $\C$-algebra $R$ of finite type. Then there is a mixed Hodge structure over $R$ on $H^k(Y^\an,\R)$.
\end{custom}	

Thus, by definition there are Hodge- and weight filtrations on $H^k(Y^\an,\R)\tensor~R$ such that the weight graded objects are pure Hodge structures over $R$. The associated Hodge-Weil structure is defined by the same formulas as an ordinary Hodge structure. To add the \emph{Weil} is necessary because over an Artin $\C$-algebra there is no canonical complex conjugation.

As a result of the Hodge theoretic considerations, we obtain the following result. It might be known to experts, but we did not find a proof in the literature.

\begin{custom}[Theorem \ref{thm free singular}]
Let $R$ be a local Artin $\C$-algebra with residue field $\C$,  let $S=\Spec R$ and let $Y$ be a proper, simple normal crossing $\C$-variety. 
Assume that $f:\sY \to S$ is a locally trivial deformation of $Y$ and that $g:\sX\to S$ is smooth and proper. Let $i:\sY \to \sX$ be an $S$-morphism. Then for all $p,q$ the morphism $i^*:R^qg_*\Omega^p_{\sX/S} \to R^qf_*\tOm^p_{\sY/S}$ has a free cokernel.
\end{custom}	

My interest in the subject arose from applications to deformations of Lagrangian subvarieties of symplectic manifolds, see \cite{CL11}, where this techniques are applied in the study deformations of singular Lagrangian subvarieties of symplectic manifolds. 

Let us spend some words about the structure of this article. 
In section \ref{sec defo} we recall the definition of \emph{locally trivial} deformations in the Zariski and analytic context. The theory of Weil restriction as presented in section \ref{sec weil} relates Hodge- and Hodge-Weil structures. Its exploitation in the infinitesimal setup is the main new feature of this work and its motivation is purely geometric. It is seen as an formalization of the process of regarding a complex manifold as a differentiable manifold.

Mixed Hodge structures and mixed Hodge-Weil structures over local Artin $\C$-algebras are introduced in section \ref{sec hodge weil}. In combination with commutative algebra they are essentially used in the proof of Theorem \ref{thm free singular}. Hodge structures over Artinian bases are intermediate objects between ordinary Hodge structures and variations of Hodge structures. Hodge-Weil structures are a tool to transport certain features, especially complex conjugation and Hodge decomposition, to the infinitesimal setup. 

Section \ref{sec mhs} provides a construction of a mixed Hodge structure over a local Artin $\C$-algebra $R$ on the cohomology of simple normal crossing varieties over $S=\Spec R$.

\section*{Notations and conventions}\label{subsec notation}

We denote by $k$ an algebraically closed field of characteristic zero. 
$\Set$ is the category of sets, $\Schm$ the category of schemes. For a scheme $Z$ the category of schemes over $Z$ is denoted by $\Sch{Z}$.

The term \emph{algebraic variety} will stand for a separated reduced $k$-scheme of finite type. %For an Artin ring $R$ we do not distinguish between a quasi-coherent sheaf on $S=\Spec R$ and its $R$-module of global sections. 
A $k$-variety $Y$ of equidimension $n$ is called a \emph{normal crossing variety} if for every closed point $y\in Y$ there is an $r\in\N_0$ such that $\widehat{\sO}_{Y,y} \isom k[[ y_1,\ldots,y_{n+1}]]/(y_1\cdot\ldots\cdot y_r)$. It is called a \emph{simple normal crossing variety} if in addition every irreducible component is nonsingular. 

For a $\C$-scheme $X$ of finite type we write $X^\an$ for the associated complex space. For a quasi-coherent $\sO_X$-module $F$ we denote by $F^\an$ the associated $\sO_{X^\an}$-module $\vphi^*F$ where $\vphi:X^\an \to X$ is the canonical morphism of ringed spaces. 

\subsection*{Acknowledgements} This work is part of the author's thesis. I would like to thank my advisor Manfred Lehn for his support and his generosity in sharing insights. Moreover, I am very grateful to Duco van Straten for the subliminal conveyance of very important ideas and to Stefan M\"uller-Stach and Claire Voisin for helpful discussions.
While working on this project, I benefited from the support of the DFG through the SFB/TR 45 ``Periods, moduli spaces and arithmetic of algebraic varieties'', 
the CNRS and the Institut Fourier. 

\section{Locally trivial deformations}\label{sec defo}

We recall the definition of locally trivial deformations, for a detailed exposition see \cite{Se}.
By $\Art_k$ we denote the category of local Artinian $k$-algebras with residue field $k$. The maximal ideal of an element $R\in \Art_k$ will be denoted by $\gothm$. 

\begin{Def}\label{defo sec schemes} 
Let $X$ be a $k$-scheme or let $k=\C$ and $X$ be a complex space. The functor
\[
D_{X}:\Art_k \to \Set, \quad R\mapsto \left\{\textrm{deformations of } X \textrm{ over } S=\Spec R\right\}/\sim
\]
where $\sim$ is the relation of isomorphism, is called \emph{functor of deformations of $X$}. A deformation
\[
 \xymatrix{X \ar[d]\ar@{^(->}[r] & \sX \ar[d]\\
0 \ar[r] & S
}\]
of $X$ over $S=\Spec R$, $R\in \Art_k$, is called \emph{(Zariski resp. analytically) locally trivial}, if for every $x\in X$ there is an open neighbourhood $x \in U \subset X$ in the Zariski- resp. Euclidean topology and an $S$-isomorphism $\sX\vert_U \to[\isom] X\vert_U \times S$ restricting to the identity on the central fiber. 
\end{Def}

Recall \cite[Exc II.8.6]{Ha77} that for a regular $k$-scheme every deformation is locally trivial.
%------------------------------------------------------------------------------------------

\begin{Def}\label{defo sec morphisms} 
Let $i:Y\to X$ be a morphism of algebraic $k$-schemes or complex spaces, let $R\in \Art_k$ and $S=\Spec R$, and let $I:\sY\to \sX$ be a deformation of $i$ over $S$. It is called \emph{(Zariski resp. analytically) locally trivial} if for every $x\in X$, $y\in Y$ with $i(y)=x$ there are open subsets $U\subset X$, $V\subset Y$ in the Zariski- resp. Euclidean topology with $y \in V$, $i(V)\subset U$ and an isomorphism
\[
\begin{xy}
\xymatrix{
\sX_{\vert U} \ar[dr]\ar[rr]^\isom && {X}_{\vert U} \times_k S \ar[dl]\\
&S&\\
\sY_{\vert V} \ar[uu]^{I_{\vert V}}\ar[rr]^\isom\ar[ur] && {Y}_{\vert V} \times_k S \ar[uu]_{{i}_{\vert V}\times_k \id}\ar[ul]\\
}
\end{xy}
\]
In other words, $I:\sY\to \sX$ induces the trivial deformation on $V$ and $U$. The functor
\[
D^\lt_{i}:\Art_k \to \Set, \quad R\mapsto \left\{\textrm{locally trivial deformations of } i \textrm{ over } S\right\}/\sim
\]
where $\sim$ is the relation of isomorphism, is called the \emph{functor of locally trivial deformations of $i$}. 
\end{Def}
%------------------------------------------------------------------------------------------
Sometimes the focus is not on the central fiber, but on the morphism. For the sake of simplicity we introduce the following terminology.
\begin{Def}\label{definition locally trivial}
Let $R \in \Art_k$ and let $S=\Spec R$. A \emph{locally trivial scheme} is a morphism $\sX \to S$ of schemes which is a locally trivial deformation of its central fiber $X:=\sX \times_{S} \Spec R/\gothm$. 
Similarly, we say that a locally trivial scheme $\sX \to S$ is a \emph{locally trivial variety}, \emph{locally trivial simple normal crossing variety}   respectively \emph{locally trivial normal crossing variety}  if $X$ is a variety, a normal crossing variety respectively a simple normal crossing variety .
\end{Def}

%------------------------------------------------------------------------------------------
\section[Weil restriction]{Weil restriction}\label{sec weil}
%------------------------------------------------------------------------------------------

%------------------------------------------------------------------------------------------
For our Hodge theoretical considerations we need the theory of Weil restriction as an essential tool. The foundations of this theory were laid by Grothendieck in \cite{Gro59,Gro60}. In our case it boils down to associating an $\R$-scheme $S_\wls$ with a $\C$-scheme $S$ such that the $\R$-valued points of $S_\wls$ are exactly the $\C$-valued points of $S$. Technically, this is phrased in the language of functors and representability. However, in this particular case we interpret Weil restriction simply as the algebro-geometric analogue of the process of regarding a complex manifold as a differentiable manifold.

We extend the concept of Weil restriction to modules. We are not aware that this has been done systematically before. Nevertheless, it is an elementary byproduct of the functorial treatment. We prove some comparison results between $R$-modules and their Weil restrictions.

%------------------------------------------------------------------------------------------
\subsection{Weil restriction}\label{subsec weil}
%------------------------------------------------------------------------------------------
Let $S\stackrel{f}{\rightarrow}Z\stackrel{p}{\rightarrow}W$ be morphisms of schemes and consider the functor
\begin{equation}\label{weil functor}
\wlfun{S}{Z}{W}:\left(\Sch{W}\right)^\op\to \Set, \quad S'\mapsto \Mor_{\Sch{Z}}(S'\times_W Z,S).
\end{equation}
In fact, we have $\Mor_{\Sch{S'\times_W Z}}(S'\times_W Z,S\times_W Z)=\Mor_{\Sch{Z}}(S'\times_W Z,S)$, which follows from the universal property of the fiber product. Therefore, the functor $\wlfun{S}{Z}{W}$ coincides with the one defined by Grothendieck in \cite[C.2, pp.12]{Gro59}. The functor $\wlfun{S}{Z}{W}$ is representable in the following cases.
\begin{enumerate}
	\item If $S\to Z$ is proper and flat and $S\to W$ is quasiprojective, this functor is representable by \cite[4.c., p.20]{Gro60} by a $W$-scheme $S_\wl$. 
	\item Suppose that $Z\to W$ is finite and locally free, i.e. finite, flat and of finite presentation, and that moreover for each $x\in W$ and each finite set of points $P\subset S\times_W k(x)$ there is an affine open $U\subset S$ containing $P$. Then $\wlfun{S}{Z}{W}$ is representable by a $W$-scheme $S_\wl$ by \cite[7.6, Thm 4]{BLR90}.
\end{enumerate}
The $W$-scheme $S_\wl$ is called the \emph{Weil restriction} of $S$.
%------------------------------------------------------------------------------------------
\subsection{Properties of Weil restriction}\label{subsec prop weil}
%------------------------------------------------------------------------------------------
We will collect some properties of the process of Weil restriction. If not otherwise stated, proofs are found in \cite[Ch 7.6]{BLR90}. Recall that a presheaf of sets on $\Sch{Z}$ is a functor $\left(\Sch{Z}\right)^\op \to \Set$. The category of presheaves of sets on $\Sch{Z}$ is denoted by $\Psh(Z)$. By the Yoneda embedding a $Z$-scheme $S$ may be interpreted as a prescheaf of sets on $\Sch{Z}$ via
\[
\ul{S}: \left(\Sch{Z}\right)^\op \to \Set, \quad T \mapsto \Mor_\Sch{Z}(T,S).
\]
We will not distinguish between $S$ and $\ul{S}$. Pushforward of presheaves along the morphism $p:Z\to W$ is the functor
\[
p_*:\Psh(Z) \to \Psh(W), \quad F\mapsto \left(S' \mapsto F(S'\times_W Z)\right)
\]
and it coincides with Weil restriction on the full subcategory $\Sch{Z}$, i.e. $S_\wl = p_* S$. To be represented by $S_\wl$ means that
\begin{equation}\label{weil restriction}
\Mor_{\Sch{Z}}(S'\times_W Z,S) = \Mor_{\Sch{W}}(S',S_\wl).
\end{equation}
In other words, $S\mapsto S_\wl$ is right adjoint to the pullback $S'\mapsto p^*S'=S'\times_W Z$. In particular, for a $Z$-scheme $S$ there is a canonical morphism $\eta:S_\wl \times_W Z \to S$. If $p:Z\to W$ is proper, flat and of finite presentation, then $p_*$ preserves open and closed immersions. 

We will now specialize to $Z=\Spec \C$ and $W=\Spec \R$. In this case every quasi-projective $\C$-scheme $S$ has a Weil restriction. We write $S_\wls$ instead of $S_\wl$. The functor $p_*$ sends affine schemes to affine schemes, in other words, $p_*S$ is representable by an affine scheme $S_\wls$. If $S=\Spec R$ we will write $R_\wl$ for the coordinate ring of $S_\wl$. Equation \eqref{weil restriction} in particular gives $S(\C)=S_\wl(\R)$. If $S=\Spec R$, the morphism $\eta$ from the adjointness property gives a ring homomorphism $\eta:R\to~R_\wls \otimes_\R\C$.

Let $S=\cup_i U_i$ be a covering by open affine subschemes, such that for given $t_1,t_2\in S$ there is an index $i_0$ with $t_1,t_2\in U_{i_0}$. The proof of representability in \cite[7.6, Thm 4]{BLR90} shows that under this assumption the $\left(U_i\right)_{\wls}$ will cover $S_\wls$. For $R = \C[z_1,\ldots,z_n]/(f_1,\ldots,f_k)$ we have
\begin{equation}\label{weil finite type}
R_\wls = \R[x_1,y_1,\ldots,x_n,y_n]/(g_1,h_1,\ldots,g_k,h_k)
\end{equation}
where $f_j=g_j+ih_j$ if we evaluate at $z_k=x_k+i y_k$. 

If we define $\ba{S}:=S\times_\sigma \C$ where $\sigma:\C\to \C$ is the complex conjugation then \eqref{weil restriction} tells us that there is a canonical isomorphism $\ba{S}_\wls \cong S_\wls$ and by \cite[Ch 1, 4.11.3]{Sch94} there is a canonical isomorphism $S_\wls\times_\R \C \to S\times_\C \ba{S}$ such that $\eta$ is identified with projection on the first factor. In particular, $\eta$ is faithfully flat as the projection $S\times_\C \ba{S} \to S$ is faithfully flat.
%------------------------------------------------------------------------------------------
\begin{Lem}\label{lemma weil preserves local artin}
If $R$ is a local Artin $\C$-algebra with residue field $\C$, then $R_\wls$ is a local Artin $\R$-algebra with residue field $\R$.
\end{Lem}
\begin{proof}
By \eqref{weil finite type} we see that $R_\wls$ is an $\R$-algebra of finite type. A maximal ideal $\gothm \subset R_\wls$ will define a homomorphism $R_\wls\to R_\wls/\gothm = k$, where $k$ is a finite field extension of $\R$ by Hilbert's Nullstellensatz. So $k=\R$ or $\C$. By the defining property of Weil restriction we have $\Hom_\R(R_\wls,\R)=\Hom_\C(R,\C)$ and $\Hom_\R(R_\wls,\C) = \Hom_\C(R,\C\otimes_\R \C)=\Hom_\C(R,\C\times \C)$ both of which consist of one element. But the composition of the morphism $R\to\R$ with the inclusion $\R \subset \C$ is the unique morphism $R\to\C$. Thus, $R_\R$ is a local ring with unique maximal ideal $\gothm$ and residue field $\R$. As $R_\wls$ is of finite type, $R_\wls=P/I$ where $P$ is a polynomial ring and $I\subset P$ an ideal. The preimage $\gothn$ of $\gothm$ under the natural map $P\to R_\wls$ is the unique maximal ideal of $P$ containing $I$. Let $I\subset \gothp \subset \gothn$ be a minimal prime ideal containing $I$. As $P$ is a Jacobson ring by the general form of the Nullstellensatz, see \cite[Thm 4.19]{E}, the ideal $\gothp$ is the intersection of maximal ideals, so that $\gothp=\gothn$. Taking a primary decomposition of $I$ we see that $\gothn^k \subset I$ for some $k$, so $R_\wls=P/I$ is Artinian.
\end{proof}
%------------------------------------------------------------------------------------------
\begin{Def}\label{def weil restriction}
Let $S$ be a $\C$-scheme, $F$ be a quasi-coherent sheaf of $\sO_S$-modules, denote by $q:S_\wls \times_\R\C \to S_\wls$ the canonical projection and let $\eta: S_\wls \times_\R\C \to S$ be as in \ref{subsec prop weil}. We define the $S_\wls$-module $$F_\wls:= q_* \eta^* F$$ and call it the \emph{Weil restriction} of $F$.
\end{Def}
%------------------------------------------------------------------------------------------
If $S=\Spec R$ and $M$ is an $R$-module, then $M_\wls = M \otimes_R (R_\wls \otimes_\R\C)$ considered as an $R_\wls$ module. In the special case $M=H\otimes_\C R$ for some $\C$-vector space $H$, we find $M_\wls =H \otimes_\R R_\wls$. Weil restriction for modules has the following useful property.
%------------------------------------------------------------------------------------------
\begin{Lem}\label{lem weil restriction}
The functor $F\mapsto F_\wls$ is faithfully exact, i.e. the sequence $K' \to K \to K''$ is exact if and only if $K_\wls' \to K_\wls \to K_\wls''$ is exact.
\end{Lem}
\begin{proof}
The morphism $\eta$ is faithfully flat as noted at the end of section \ref{subsec prop weil}. Therefore, $\eta^*$ is faithfully exact. Also $q_*$ is faithfully exact, as $q$ is affine.
\end{proof}
%------------------------------------------------------------------------------------------
\begin{Lem}\label{lemma coker}
Let $(R,\gothm)$ be a local Artin $\C$-algebra and $F$ be a finitely generated $R$-module. Then $F$ is a free $R$-module if and only if $F_\wls$ is a free $R_\wls$-module.
\end{Lem}
\begin{proof}
We will argue separately for $\eta^*$ and $q_*$. For brevity we write $(R',\gothm')$ instead of $(R_\wls\tensor_\R\C,\gothm_\wls\tensor_R\C)$. Clearly, $\eta^*F = F\tensor_R R'$ is free if $F$ is. Suppose $\eta^*F$ is free. We take a minimal set of generators for $F$ and obtain a surjection $\vphi:~R^n\to~F$ for some $n$. By Nakayama's Lemma $n=\dim_\C F\tensor_R R/\gothm$ and as $F\tensor_R R' \tensor_{R'} R'/\gothm' = F\tensor_R {R/\gothm} \tensor_{R/\gothm} R'/\gothm'$ this is the rank of $\eta^*F$. But as $\eta^*$ is faithfully exact, $\eta^*\ker\vphi=\ker \eta^*\vphi = 0$. So $\ker\vphi =0$ and $F$ is free.

Let $F'$ be an $R'$-module. If $F'$ is free as an $R'$-module, then it is free as an $R_\wls$-module, for $R'$ is free over $R_\wls$. Suppose $F'$ is free as an $R_\wls$-module. Since $F'$ is an $R'=R_\wls \tensor_\R\C$-module, the submodule $\gothm_\wls F'$ is a $\C$-vector space. Thus $\gothm_\wls F' = \gothm'F'$. If we take $x_1,\ldots,x_k \in F'$ whose residue classes modulo $\gothm_\wls$ form a $\C$-basis of $F'/\gothm_\wls F'$, then $F$ is freely generated over $R_\wls$ by $x_1, ix_1, \ldots, x_k, ix_k$. In other words, $F$ is freely generated over $R'$ by $x_1, \ldots, x_k$. So $F'$ is a free $R'$-module.
\end{proof}
%------------------------------------------------------------------------------------------
\begin{Ex}
For the projective space $S=\P^1_\C$ of lines in $\C^2$ one finds that $S_\wls$ is isomorphic over $\R$ to the quadric $Q$ in $\P^3_\R$ given by $$x_1x_2 - x_0^2 - x_3^2 = 0.$$ see \cite[Example II.1.6]{mydiss}. As an illustration of this claim note that the map
\[
 Q(\R)\to \S^2 \subset \R^3,\; [x_0:x_1:x_2:x_3] \mapsto \frac{1}{x_1+x_2} (x_1-x_2, 2x_0, 2x_3)
\]
is an isomorphism so that indeed $\P^1_\C(\C) = Q(\R)$.
\end{Ex}
%------------------------------------------------------------------------------------------

%------------------------------------------------------------------------------------------
\section{Hodge-Weil theory}\label{sec hodge weil}
%------------------------------------------------------------------------------------------
We introduce the notion of a \emph{mixed Hodge structure over $R$}, where $R$ is a local Artin $\C$-algebra with residue field $\C$. 
The cohomology with coefficients in the constant sheaf $R$ of a locally trivial simple normal crossing variety over  $\Spec R$ carries such a structure, see Theorem \ref{thm mixed}.

The purpose of this concept is to carry out Hodge theoretic arguments infinitesimally and it plays a central role in the proof of Theorem \ref{thm free singular}. As far as we know, it has not been studied before. The problem for $R\neq\C$ is that there is no analogue of the complex conjugation on the underlying $R$-module $H$. We will cure this by introducing the notion of a \emph{mixed Hodge-Weil structure over $R'$}, where $R'$ is now a local Artin $\R$-algebra with residue field $\R$. This notion is a formalization of the Weil restriction of a mixed Hodge structure over $R$ and there is canonically a complex conjugation.
%------------------------------------------------------------------------------------------
\begin{Def}\label{definition mhs over r}
Let $R$ be a local Artin $\C$-algebra with residue field $\C$. A \emph{mixed Hodge structure over $R$} is a triple $\sH=(H_\R, F^\bullet, W_\bullet)$, which consists of a finite dimensional $\R$-vector space $H_\R$ and two filtrations $F^\bullet$ and $W_\bullet$ on $H:=(H_\R\otimes_\R \C)\tensor_\C R$. These are a finite decreasing filtration 
\[
H \supset \ldots \supset F^p \supset F^{p+1} \supset \ldots \supset 0
\]
and a finite increasing filtration 
\[
0 \subset \ldots \subset W_m \subset W_{m+1} \subset \ldots \subset H
\]
satisfying the following properties.
\begin{enumerate}
	\item All graded objects $\Gr_F^p \Gr_m^W H$ are free $R$-modules.
	\item The fiber $\sH\tensor_R \C=(H_\R \tensor_\R \C,F^\bullet\otimes_R \C, W_\bullet\otimes_R\C)$ over the unique point of $S=\Spec R$ is a mixed Hodge structure.
\end{enumerate}
Note that condition (1) implies that the $W_m$ and the $F^p$ are free $R$-modules. We will also call $\sH\tensor_R \C$ the \emph{central fiber} of $\sH$. In case $\sH\tensor_R \C$ is a pure Hodge structure of weight $k$, we call $\sH$ a \emph{pure Hodge structure} over $R$ of weight $k$.
\end{Def}
%------------------------------------------------------------------------------------------
\begin{Def}
Let $R$ be a local Artin $\C$-Algebra and $\sH=(H_\R,F,W)$, $\sH'=(H_\R',F',W')$ be mixed Hodge structures over $R$. A \emph{morphism of mixed Hodge structures over $R$} is a linear map $f_\R:H_\R\to H_\R'$ such that the induced morphism $f=f_\R\tensor \id_R : H\to H'$ preserves both filtrations, i.e. $f(F^p)\subset {F^p}'$ and $f(W_m)\subset {W_m}'$. Here again $H=(H_\R\otimes_\R \C)\tensor_\C R$ and $H'$ is defined analogously. We will often call $f$ instead of $f_\R$ a morphism of mixed Hodge structures over $R$ when there is no danger of confusion.
\end{Def}
%------------------------------------------------------------------------------------------
\begin{Rem}---
 
\begin{enumerate}
	\item If $\sH=(H_\R,F,W)$ is a pure Hodge structure of weight $k$ over an Artin ring $R$, then Nakayama's Lemma implies that $W$ is a trivial filtration, i.e. $H=W_k \supset W_{k-1} = 0$. We will therefore suppress $W$ in the notation and speak of a pure Hodge structure $\sH=(H_\R,F)$ over $R$.
	\item There is a complex conjugation $H_\R\tensor_\R\C\to H_\R\tensor_\R\C$ defined by $\ba{h\tensor \lambda}:=h\tensor\ba{\lambda}$. However this does not canonically extend to an $\R$-linear map $H\to H$, as $H$ is a tensor product over $\C$ and complex conjugation is only $\R$-linear.
\end{enumerate}
\end{Rem}
%------------------------------------------------------------------------------------------
The notion of a Hodge structure over $R$ is an infinitesimal version of a variation of Hodge structures. The problem in replacing the base manifold $S$ of the variation with a local Artin ring $R$ is that $S=\Spec R = \{ \gothm \}$ and just posing the condition that the fiber over $R/\gothm$ be a mixed Hodge structure is not enough. The (pointwise) complex conjugates $\overline{F^p}$ of the Hodge filtration of a variation of Hodge structures do not in general form holomorphic vector bundles in case $S$ is a complex manifold, so there is no algebraic incarnation of $\overline{F^p}$. As a substitute we introduce the following notion.
%------------------------------------------------------------------------------------------
\begin{Def}\label{definition mhws over r}
Let $R$ be a local Artin $\R$-algebra with residue field $\R$. A \emph{mixed Hodge-Weil structure} over $R$ is a triple $\sH=(H_\R, F^\bullet, W_\bullet)$, which consists of a finite dimensional $\R$-vector space $H_\R$ and two filtrations $F^\bullet$ and $W_\bullet$ on $H:=(H_\R\otimes_\R \C)\tensor_\R R$. These are a finite decreasing filtration 
\[
H \supset \ldots \supset F^p \supset F^{p+1} \supset \ldots \supset 0
\]
and a finite increasing filtration 
\[
0 \subset \ldots \subset W_m \subset W_{m+1} \subset \ldots \subset H
\]
satisfying the following properties.
\begin{enumerate}
	\item All graded objects $\Gr_F^p \Gr_m^W H$ are free $R$-modules.
	\item The fiber $\sH\tensor_R \R=(H_\R \tensor_\R \C,F^\bullet\otimes_R \R, W_\bullet\otimes_R\R)$ over the unique point of $S=\Spec R$ is a mixed Hodge structure.
\end{enumerate}
Note that as in Definition \ref{definition mhs over r}, condition (1) implies that the $W_m$ and the $F^p$ are free $R$-modules. We will also call $\sH\tensor_R \R$ the \emph{central fiber} of $\sH$. In case $\sH\tensor_R \C$ is a pure Hodge structure of weight $k$, we call $\sH$ a \emph{pure Hodge-Weil structure} over $R$ of weight $k$.
\end{Def}
%------------------------------------------------------------------------------------------
\begin{Def}
Let $R$ be a local Artin $\R$-algebra with residue field $\R$ and $\sH=(H_\R,F,W)$, $\sH'=(H_\R',F',W')$ be mixed Hodge-Weil structures over $R$. A morphism of mixed Hodge-Weil structures over $R$ is a linear map $f:H_\R\to H_\R'$ such that the induced morphism $f_R=f\tensor \id_R : H\to H'$ preserves both filtrations, i.e. $f_R(F^p)\subset {F^p}'$ and $f_R(W_m)\subset {W_m}'$. Here again $H=(H_\R\otimes_\R \C)\tensor_\R R$ and $H'$ is defined analogously. We will write $f$ instead of $f_R$ when there is no danger of confusion.
\end{Def}
%------------------------------------------------------------------------------------------
\begin{Rem}---
 
\begin{enumerate}
	\item As in the Hodge-case, we write $\sH=(H_\R,F)$ for a pure Hodge-Weil structure.
	\item The complex conjugation $H_\R\tensor_\R\C\to H_\R\tensor_\R\C$ extends canonically to an $\R$-linear map $H\to H$. Since morphisms of mixed Hodge-Weil structures are defined over $\R$, they are compatible with complex conjugation.
\end{enumerate}
\end{Rem}
%------------------------------------------------------------------------------------------
Recall that $R\in\Art_\C$ the ring $R_\wls$ is a local Artin $\R$-algebra with residue field $\R$ by Lemma \ref{lemma weil preserves local artin}. Therefore, the statement of the following Lemma makes sense.
%------------------------------------------------------------------------------------------
\begin{Lem}\label{lemma weil hodge = hodge weil}
Let $\sH=(H_\R, F^\bullet, W_\bullet)$ be a mixed Hodge structure over a local Artin $\C$-Algebra $R$. Then $\sH_\wls=\left(H_\R, F_\wls^\bullet, {\left(W_\wls\right)}_\bullet\right)$ is a mixed Hodge-Weil structure over $R_\wls$ and the central fibers of $\sH$ and $\sH_\wls$ are isomorphic as mixed Hodge structures. Moreover, the Weil restriction of a morphism of mixed Hodge structures is a morphism of mixed Hodge-Weil structures.
\end{Lem}
\begin{proof} The remark after Definition \ref{def weil restriction} tells us that
\begin{equation}\label{weil restricted mhs}
H_\wls = (H_\R \tensor_\R\C \tensor_\C R)_\wls = (H_\R \tensor_\R\C) \tensor_\R R_\wls.
\end{equation}
By Lemma \ref{lem weil restriction} we see that the $F^p_\wls$ and ${\left(W_m\right)}_\wls$ are submodules of $H_\wls=(H_\R \tensor_\R\C) \tensor_\R R_\wls$. By Lemma \ref{lemma coker} the modules $\left(\Gr_F^p \Gr_m^W H\right)_\wls$ are free and by Lemma \ref{lem weil restriction} they are the graded objects of the filtrations $F^p_\wls$ and ${\left(W_m\right)}_\wls$. Let $\gothm'$ be the maximal ideal of $R_\wls$. As $R_\wls/\gothm' = \R$ we see from \ref{weil restricted mhs} that $H_\wls\tensor_\R R_\wls/\gothm' = H_\R \tensor_\R\C$. For the same reason $F^p_\wls \tensor \R = F^p \tensor \C$ and $(W_m)_\wls \tensor \R = W_m \tensor \C$ so that $\sH_\wls\tensor \R$ is a mixed Hodge structure. The proof also shows the statement about the central fibers and the statement about morphisms is immediate from the functoriality of the Weil restriction.
\end{proof}
%------------------------------------------------------------------------------------------
\begin{Lem}\label{lemma hodge decomposition}
Let $R$ be a local Artin $\R$-Algebra with residue field $\R$ and $\sH=(H_\R, F^\bullet)$ a pure Hodge-Weil structure of weight $k$. Then
\begin{align}
\label{h=f+fbar}H&=F^p\oplus \ba{F^{q+1}}, \qquad \forall p,q, p+q=k,\\
\label{h=hpq}H&=\bigoplus_{p+q=k}H^{p,q}, \qquad H^{p,q}=F^p\cap \ba{F^q} \qquad \textrm{ and}\\
\label{f=hpq} F^p&=\bigoplus_{r\geq p} H^{r,k-r}.
\end{align}
In particular, the last statement implies that the $H^{p,q}$ are free and lift the subquotients $\Gr_F^p H$ to subobjects of $H$.
\end{Lem}
\begin{proof}
As $\sH\tensor_R\R$ is a pure Hodge structure, we have
\[
H\tensor_R\R=F^p\tensor_R\R\oplus \ba{F^{q+1}\tensor_R\R} \quad \forall p,q, p+q=k.
\]
Hence \eqref{h=f+fbar} follows from Nakayama's Lemma. Now \eqref{h=f+fbar} implies \eqref{h=hpq} just as in the case of ordinary Hodge structures. We will recall the proof. Let $\alpha \in F^p \subset H$ and write $\alpha=\beta + \gamma$ where $\beta \in F^{p+1}$, $\gamma \in \ba{F^{k-p}}$ according to $H=F^{p+1}\oplus \ba{F^{k-p}}$. Then $\gamma=\alpha-\beta \in F^p \cap \ba{F^{k-p}}=H^{p,k-p}$. This shows that $F^p=F^{p+1}\oplus H^{p,q}$, and \eqref{h=hpq} and \eqref{f=hpq} follow by induction on $p$.
\end{proof}
%------------------------------------------------------------------------------------------
\begin{Lem}\label{lemma morphisms}
Let $R$ be a local Artin $\C$-Algebra, let $\sH=(H_\R,F,W)$ and $\sH'=(H_\R',F',W')$ be mixed Hodge structures over $R$ and let $f:H \to H'$ be a morphism of mixed Hodge structures over $R$. Then $f^{p,q}:=f\vert_{H^{p,q}}$ satisfies $f^\pq\left(H^{p,q}\right)\subset\left(H'\right)^{p,q}$ and $f=\sum_{p,q} f^{p,q}$. Moreover, all $f^{p,q}$ have constant rank in the sense of Definition \ref{definition constant rank}.
\end{Lem}
\begin{proof}
By \eqref{h=hpq} the image of $f^{p,q}$ is contained in $(H')^{p,q}$, because $f$ is defined over $\R$ and preserves the Hodge filtration. Again, as $f$ is defined over $\R$ its cokernel is $\coker f = \coker \left(f_\R : H_\R \to H_\R'\right)\tensor_\R R$, so it is free. Then $$\coker f=\bigoplus_{p,q} \coker f^{p,q}$$ implies that $\coker f^{p,q}$ is free. So the claim follows from Lemma \ref{lemma constant rank}.
\end{proof}
%------------------------------------------------------------------------------------------

%------------------------------------------------------------------------------------------
\section[MHS for normal crossing varieties]{Mixed Hodge structures for normal crossing varieties}\label{sec mhs}
%------------------------------------------------------------------------------------------
Let $S=\Spec R$ where $R\in \Art_\C$ and let $f:\sY\to S$ be a proper, locally trivial simple normal crossing $\C$-variety. We will construct a complex $\tOm_{\sY/S}^\bullet$, which calculates the cohomology with coefficients in the constant sheaf $\ul{R}_{\sY^\an}$ on $\sY^\an$. 

Using the complex $\tOm_{\sY/S}^\bullet$ and its canonical resolution, we contruct a mixed Hodge structure over $R$ on $H^k(\sY^\an,\ul R_{\sY^\an})$.
%------------------------------------------------------------------------------------------
\begin{Def}\label{definition tom}
Let $S=\Spec R$ where $R\in \Art_k$ and let $f:\sY\to S$ be a locally trivial variety in the sense of Definition \ref{definition locally trivial}. We define $\tau^k_{\sY/S} \subset \Omega^k_{\sY/S}$ to be the subsheaf of sections whose support is contained in the singular locus of $f$. 
We put $\tOm^k_{\sY/S}:=\Omega^k_{\sY/S}/\tau^k_{\sY/S}$. 
\end{Def}
%------------------------------------------------------------------------------------------
If $\alpha \in \Omega_{\sY/S}^k$ vanishes on $Y^\reg$, then $d\alpha$ vanishes there as well. Therefore, $\tau^\bullet_{\sY/S} \subset \Omega_{\sY/S}^\bullet$ is a subcomplex and $\tOm_{\sY/S}^\bullet$ is a complex. Next we will show that irreducible components of a variety extend to flat subschemes on locally trivial deformations. This will take some commutative algebra.
%------------------------------------------------------------------------------------------
\begin{Lem}\label{lemma annihilator} 
Let $A$ be a reduced noetherian ring and $\gothp_1,\ldots, \gothp_n$ be the pairwise distinct minimal prime ideals of $A$. Then $\Ann \gothp_j = \cap_{i\neq j}\gothp_i$ for each~$j$.
\end{Lem}
\begin{proof}
Let $A_i = A/\gothp_i$ and $\phi : A\to A_1\times\ldots\times A_n$ be the canonical map. It is injective, because $\cap_i\gothp_i = \nil(A) = 0$. Suppose $a\in \cap_{i\neq j}\gothp_i$, $b\in \gothp_j$ and write $\phi(a)=(a_1,\ldots,a_n)$ and $\phi(b)=(b_1,\ldots,b_n)$. Then $\phi(ab)=(a_1b_1,\ldots, a_nb_n)=0$ because $a_i=0$ for $i\neq j$ and $b_j=0$. But $\phi$ is injective, hence $ab=0$, in other words, $a\in \Ann\gothp_j$, so $\Ann \gothp_j \supset \cap_{i\neq j}\gothp_i$.

Let $a\in \Ann \gothp_j$. Then for every $b\in \gothp_j$ we have $0=\phi(ab)=(a_1 b_1,\ldots,a_n b_n)$ in the above notation, where $b_j=0$. As the $\gothp_i$ are minimal and pairwise distinct, $\gothp_j\backslash \gothp_k\neq \emptyset$ for every $k\neq j$. If we fix $k$ and choose $b \in \gothp_j\backslash \gothp_k$, then $b_k\neq 0$. So $a_k b_k = 0$ implies that $a_k = 0$ as $A_k$ is an integral domain, so $a\in \gothp_k$. Choosing different $b$ we see that $a \in \cap_{i\neq j}\gothp_i$ completing the proof.
\end{proof}
%------------------------------------------------------------------------------------------
\begin{Lem}\label{lemma connecting map zero}
Let $A$ be a reduced noetherian ring, $\gothp \subset A$ be a minimal prime ideal and $\psi: \gothp \to A/\gothp$ be an $A$-module homomorphism. Then $\psi=0$.
\end{Lem}
\begin{proof}
Let $\gothp, \gothp_1, \ldots, \gothp_n$ be the pairwise distinct minimal prime ideals of $A$ and $N:=\img \psi \subset A/\gothp$. We will show that $N=0$. By Lemma \ref{lemma annihilator} we have $\Ann \gothp = \cap_i \gothp_i$. So $\gothp \notin \supp(\gothp)=V(\Ann \gothp)$, for otherwise $\cap_i \gothp_i \subset \gothp$ and thus $\gothp_i\subset \gothp$ for some $i$ as $\gothp$ is prime, contradicting the fact that $\gothp\neq\gothp_i$ and $\gothp$ is minimal. Thus, $\gothp \tensor_A A_\gothp=0$ and the surjection
\[
\begin{xy}
\xymatrix{
0=\gothp \tensor_A A_\gothp \ar@{->>}[r]&  N \tensor_A A_\gothp
}
\end{xy}
\]
yields that $N_\gothp = N \tensor_A A_\gothp = 0$. Therefore, $N$ is torsion. This implies $N=0$, as it is an $A/\gothp$-submodule of the torsion-free module $A/\gothp$.
\end{proof}
%------------------------------------------------------------------------------------------
\begin{Lem}\label{lemma unique flat lifting}
Let $A$ be a reduced noetherian ring, $\gothp\subset A$ a minimal prime ideal, $R\in \Art_k$ and $\gothP\subset A\tensor_k R$ an ideal such that $A\tensor_k R/\gothP$ is a flat deformation of $A/\gothp$ over $R$. Then $\gothP=\gothp\tensor R$.
\end{Lem}
\begin{proof}
Let $\gothm \subset R$ be the maximal ideal. As $R$ is Artinian, there is $n\in \N$ such that $\gothm^{n}=0$. So we may argue inductively and assume that $\gothP/\gothm^k = \gothp\tensor R/\gothm^k\subset A\tensor R/\gothm^k$. By flatness, we obtain the commutative diagram
\begin{equation}\label{big diag}
\begin{xy}
\xymatrix@R=1.6em@C=1.35em{
& 0\ar[d] & 0\ar[d] & 0\ar[d] & \\
0 \ar[r] & \gothp \tensor \gothm^{k} /\gothm^{k+1} \ar[r] \ar[d] & A\tensor \gothm^{k}/\gothm^{k+1}  \ar[r] \ar[d] & A/\gothp \tensor \gothm^{k}/\gothm^{k+1}  \ar[r] \ar[d]^\pi & 0 \\
0 \ar[r] & \gothP /\gothm^{k+1}\ar[r] \ar[d] & A\tensor R/\gothm^{k+1} \ar[r]^(0.46)\vphi \ar[d] & A\tensor R/(\gothP +\gothm^{k+1}) \ar[r] \ar[d]^\chi & 0 \\
0 \ar[r] & \gothp \tensor R/\gothm^{k}\ar[d] \ar[r] & A\tensor R/\gothm^{k} \ar[d]\ar[r] & A/\gothp \tensor R/\gothm^{k}\ar[d] \ar[r]& 0 \\
& 0 & 0 & 0 & \\
}
\end{xy}
\end{equation}
with exact rows and columns.

If we denote the inclusion 
$\begin{xy}
\xymatrix{
\gothp \tensor R/\gothm^{k+1} \ar@{^(->}[r] & A\tensor R/\gothm^{k+1}
}
\end{xy}$
by $\psi$, then $\vphi\circ \psi$ factors as 
\[
\begin{xy}
\xymatrix@C=3em{
& A/\gothp \tensor \gothm^{k}/\gothm^{k+1} \ar[d]^\pi\\
\gothp \tensor R/\gothm^{k+1} \ar@{-->}[ru]^\exists\ar[r]_(0.4){\vphi\circ \psi} & A\tensor R/(\gothP +\gothm^{k+1})
}
\end{xy}
\]
Indeed, this can be seen as follows. Consider the commutative diagram
\[
\begin{xy}
\xymatrix@C=1.35em{
& \gothp \tensor R/\gothm^{k+1} \ar[r]^\psi \ar[d] & A\tensor R/\gothm^{k+1} \ar[r]^(0.46)\vphi \ar[d] & A\tensor R/(\gothP +\gothm^{k+1}) \ar[r] \ar[d]^\chi & 0 \\
0 \ar[r] & \gothp \tensor R/\gothm^{k} \ar[r] & A\tensor R/\gothm^{k} \ar[r] & A/\gothp \tensor R/\gothm^{k} \ar[r]& 0 \\
}
\end{xy}
\]
Then $\chi\circ \vphi\circ \psi=0$ as the bottom row is exact. Therefore, $\vphi\circ \psi$ factors through $\ker \chi$ as claimed. 

Now observe that $\gothp \tensor R/\gothm^{k+1} \to A/\gothp \tensor \gothm^{k}/\gothm^{k+1}$ is zero by Lemma \ref{lemma connecting map zero}, hence so is $\vphi \circ \psi$. Therefore $\psi$ factors through $\ker\vphi= \gothP /\gothm^{k+1}$ as
\[
\begin{xy}
\xymatrix@C=3em{
\gothp \tensor R/\gothm^{k+1} \ar@{-->}[d] \ar@{^(->}[r]^\psi & A\tensor R/\gothm^{k+1}\\
\gothP /\gothm^{k+1} \ar[ru]& 
}
\end{xy}
\]
But $\gothp \tensor R/\gothm^{k+1} \to \gothP /\gothm^{k+1}$ becomes an isomorphism after tensoring with $R/\gothm^k$, thus it is itself an isomorphism by flatness of $\gothP/\gothm^{k+1}$, see \cite[Lem A.4]{Se}.
\end{proof}
%------------------------------------------------------------------------------------------
\begin{Lem}\label{lemma irreducibles}
Let $f:\sY \to S$ be a locally trivial deformation of a reduced noetherian scheme $Y$ over an Artinian base $S=\Spec R$, $R\in \Art_k$. Then the irreducible components $Y_{\alpha}$ of $Y$ lift uniquely to subschemes $\sY_\alpha \into \sY$ flat over $S$. Moreover, each $\sY_\alpha$ is a locally trivial deformation of $Y_{\alpha}$.
\end{Lem}
\begin{proof}
Let $Y = \cup_i U_i$ be an open affine covering of $Y$ such that there are $R$-algebra isomorphisms $\theta_i:A_i\otimes_k R \to \Gamma(U_i, \sO_\sY)$ where $A_i:=\Gamma(U_i,\sO_{Y})$. An irreducible component $Y_\alpha$ of $Y$ gives a minimal prime ideal $\gothp_\alpha^i$ in each $A_i$. We define $\sY_\alpha^i$ to be the closed subscheme of $\sY\vert_U$ whose ideal is $\theta_i(\gothp_\alpha)$. Then $\sY_\alpha^i$ is a flat lifting of $Y_\alpha\vert_{U_i}$ for all $i$. Therefore, on $U_{ij}:=U_i\cap U_j$ also  $\sY_\alpha^j\vert_{U_{ij}}$ is a flat lifting of $Y_\alpha\vert_{U_{ij}}$ for all $j$. Then by Lemma \ref{lemma unique flat lifting} we conclude that $\sY_\alpha^i\vert_{U_{ij}}=\sY_\alpha^j\vert_{U_{ij}}$ and so the $\sY_\alpha^i$ are the restrictions of a closed subscheme $\sY_\alpha$ of $\sY$. The argument also shows that $\sY_\alpha$ is unique.
\end{proof}
 
\subsection{Semi-simplicial resolutions}\label{subsec ssres}
Recall that a semi-simplicial scheme $Y^{\bullet}$ is given by schemes $Y^n$ and morphisms $d^j:Y^n\to Y^{n-1}$ for $j=0, \ldots, n$ satisfying some compatibility condition. We refer to \cite[5.1]{PS} for details.

An ordinary scheme $Y$ may be considered as a trivial semi-simplicial scheme with $Y^n=Y$ and all $d^j=\id_{Y}$.
A morphism of semi-simplicial schemes $a:Y^\bullet \to Y$ from $Y^\bullet $ to an ordinary scheme is also called an \emph{augmentation} of $Y^\bullet$ to $Y$ or that $Y^\bullet$ is \emph{augmented} towards $Y$.
We will also write an augmented semi-simplicial scheme $Y^\bullet\to Y$ in the form
\[
\begin{xy}
\xymatrix@C=1em{
\ldots  \ar[r]\ar@<0.5ex>[r]\ar@<-0.5ex>[r] & Y^{1} \ar@<0.5ex>[r]\ar@<-0.5ex>[r] & Y^{0} \ar[r]& Y \\
}
\end{xy}.
\]
Dual to the notion of an semi-simplicial object is the one of a semi-cosimplicial object.
%------------------------------------------------------------------------------------------
\begin{Def}\label{def ss resolution}
Let $S$ be a $\C$-scheme and $\sY\to S$ be a proper scheme over $S$. A \emph{semi-simplicial resolution} of $\sY$ over $S$ is a semi-simplicial $S$-scheme $\sY^\bullet$ together with a morphism $a: \sY^\bullet \to \sY$ of semi-simplicial $S$-schemes such that all $a_k:\sY^k\to \sY$ are proper and $\sY^k\to S$ is smooth for all $k$.
\end{Def}
%------------------------------------------------------------------------------------------
Note that for $S=\Spec \C$ this definition does \emph{not} coincide with Deligne's \cite{De71, De74}. Deligne defines semi-simplicial resolutions for varieties over $\C$. He requieres a resolution to be of \emph{of cohomological descent}, an extra condition which he uses to construct a functorial mixed Hodge structure on the cohomology of an algebraic $\C$-variety. We do not need this here as we proof all our Hodge theoretical statements "by hand". 

\subsection{Canonical resolution for locally trivial deformations of simple normal crossing varieties}\label{subsec canres}
%------------------------------------------------------------------------------------------

Let $Y$ be a proper simple normal crossing $k$-variety and let $Y=\cup_i Y_i$ be a decomposition into irreducible components. Let $f:\sY \to S$ be a locally trivial deformation of $Y$ over $S=\Spec R$ where $R\in \Art_k$. Lemma \ref{lemma irreducibles} allows us to write 
\[
\sY = \bigcup_{i=1}^n \sY_i
\]
with flat $S$-schemes $\sY_i$. This union is a decomposition into irreducible components and $\sY_i$ is a locally trivial deformation of $Y_i$. As the $\sY_i\to S$ are flat deformations of smooth schemes, 
$\sY^{0} := \coprod_i \sY_i \to S$
is smooth as well. For a subset $I\subset [n]:=\{1, \ldots, n\}$ we put 
\begin{equation}\label{sssscheme}
\sY^I:=\bigcap_{i\in I} \sY_i, \quad \sY^{k}:=\coprod_{\left|I\right|=k+1}\sY^I.
\end{equation}

Here, by $\sY_i\cap \sY_j$ we denote the scheme $\sY_i \times_\sY \sY_j$. There exists one map $a_k:\sY^{k}\to \sY$ over $S$ and $k+1$ canonical maps $d_j:\sY^{k}\to \sY^{k-1}$ for $j=0,\ldots, k$ over $S$ coming from the $k+1$ inclusions $[k] \into [k+1]$. In other words, the collection of the $\sY^{k}$ together with the $d_j$ is a semi-simplicial $S$-scheme and the $a_k$ form an augmentation of $\sY^{\bullet}$ to $\sY$.

%------------------------------------------------------------------------------------------
\begin{Lem}\label{lemma can res}
The semi-simplicial $S$-scheme $\sY^{\bullet}$ together with the augmentation $a:\sY^{\bullet}\to \sY$ is a semi-simplicial resolution of $\sY$. We call it the \emph{canonical resolution} of $\sY$ over $S$.
\end{Lem}
%------------------------------------------------------------------------------------------
\begin{proof}
We have to show that all $\sY^{m}\to S$ are smooth morphisms. Lemma III.1.5 tells us (or rather the choice of $\sY_i$, which was made using Lemma III.1.5) that $\sY_i$ is a flat deformation of the smooth variety $Y_i$ and therefore smooth as well. For $m\geq 1$ we use that $\sY^{m}$ is a disjoint union of schemes of the form $\sY^I =  \sY_{i_0} \times_\sY  \ldots \times_\sY \sY_{i_m}$, where $I=\{i_0, \ldots, i_m\}$ and $\abs{I}=m+1$. Moreover, smoothness is a local property, so let us assume that all schemes are affine, say
\[
\sY = \Spec\sA, \quad \sY_i = \Spec\sA_i, \quad Y = \Spec A, \quad  Y_i = \Spec A_i,
\]
where $A=\sA \tensor_R k$ and $A_i = \sA_i \tensor_R k$
for $S=\Spec R$. But all morphisms $\sY_i \to \sY$ are $S$-morphisms and $\sY_i$ respectively $\sY$ are locally trivial deformations of $Y_i$ respectively $Y$. Thus, we may assume that 
$\sA_i \isom A_i \tensor_k R$ and $\sA \isom A \tensor_k R$. Note that by Lemma \ref{lemma unique flat lifting} the trivialization $\sA \isom A \tensor_k R$ already induces an isomorphism $\sA_i \isom A_i \tensor_k R$ so that we obtain an $R$-algebra isomorphism
\[
\Gamma(\sY^I,\sO_{\sY^I}) = \sA_{i_0} \tensor_\sA \ldots \tensor_\sA \sA_{i_m} \isom \left(A_{i_0} \tensor_A \ldots \tensor_A A_{i_m}\right) \tensor_k R.
\]
The ring $A_{i_0} \tensor_A \ldots \tensor_A A_{i_m}$ is the coordinate ring of the smooth $k$-variety $Y^I:= \sY^I\times_S k = Y_{i_0} \times_Y  \ldots \times_Y Y_{i_m}$. Smoothness of $Y^I$ is immediate from the normal crossing condition. This shows that also $\sY^I$ 
 is smooth over $S=\Spec R$ completing the proof.
\end{proof}
%------------------------------------------------------------------------------------------
\subsection{Semi-cosimplicial resolution for $\tOm_{\sY/S}^p$}
%------------------------------------------------------------------------------------------
For $\sY$ as in section \ref{subsec canres} the semi-simplicial $S$-scheme $\sY^{\bullet}$ induces semi-cosimplicial $\sO_\sY$-modules ${a}_*\Omega^p_{\sY^{\bullet}/S}$. The formula $\delta_{n} : = \sum_{j=0}^{n+1} (-1)^j d^j$ where $d^j=d_j^*$ makes 
\[
{a}_*\Omega^p_{\sY^{\bullet}/S}:\qquad {a_0}_*\Omega^p_{\sY^{0}/S} \to[\delta_0] {a_1}_*\Omega^p_{\sY^{1}/S} \to[\delta_1] \ldots
\]
into a complex. The augmentation $a:\sY^{\bullet}\to \sY$ induces a coagumentation 
\[
\Omega_{\sY/S}\to[a_0^*] {a_0}_*\Omega^p_{\sY^{0}/S} \to[\delta_0] {a_1}_*\Omega^p_{\sY^{1}/S} \to[\delta_1] \ldots.
\]
As $\sY^{0}\to S$ is smooth, the morphism $a_0^*$ factors through $\tOm_{\sY/S}^p$ from Definition \ref{definition tom}. Clearly, the composition $\delta_0\circ a_0^*$ is zero and we obtain a complex
\begin{equation} \label{resolution}
0\to \tau^k_{\sY/S} \to \Omega_{\sY/S}^k \to {a_0}_*\Omega^k_{\sY^{0}/S} \to {a_1}_*\Omega^k_{\sY^{1}/S} \to \ldots
\end{equation}
%------------------------------------------------------------------------------------------
All following theory is based on the important
%------------------------------------------------------------------------------------------
\begin{Lem}\label{lemma friedman}
Let $Y$ be a simple normal crossing $\C$-variety and $f:\sY \to S$ be a locally trivial deformation of $Y$ over $S=\Spec R$ with $R \in \Art_\C$. Then
\begin{enumerate}
	\item \label{eins} The sequence \eqref{resolution} is exact and so is the sequence with $\sY$ replaced by $\sY^\an$.
	\item \label{zwei} $\tOm_{\sY^\an/S}^\bullet$ is a resolution of the constant sheaf $\ul{R}_{Y^\an}$.
	\item \label{drei} The canonical map $\left(\tOm_{\sY/S}^k\right)^\an \to \tOm_{\sY^\an/S}^k$ is an isomorphism.
	\item \label{vier} The canonical map $R^i f_*\tOm_{\sY/S}^k \to R^i f^\an_*\tOm_{\sY^\an/S}^k$ is an isomorphism.
\end{enumerate}
\end{Lem}
\begin{proof} 
The question is local in $\sY$, so we may assume that $\sY=Y\times S$ is the trivial deformation. Then the resolution \eqref{resolution} is simply the pullback of the analogous resolution for $Y$ along the flat morphism $Y\times S \to Y$. This implies \eqref{eins} and \eqref{zwei}, as the respective statements are true for $Y$ by \cite[Prop 1.5]{Fr83}.

We clearly have $\left(\Omega_{\sY/S}\right)^\an \isom \Omega_{\sY^\an/S}$. Now \eqref{drei} follows from \eqref{eins} because analytification is an exact functor by \cite[Exp XII, Prop 1.3.1]{SGA1} and compatible with taking the wedge product. Moreover, \eqref{drei} implies \eqref{vier} by \cite[Exp XII, Thm 4.2]{SGA1}.
\end{proof}
%------------------------------------------------------------------------------------------
\begin{Rem} In \cite{Se55} several comparison theorems are proven for projective varieties over $\C$. A generalization of Serre's work to proper schemes of finite type over $\C$ is given in Raynaud's expos\'e \cite[Exp XII]{SGA1}. The references in the proof refer to generalizations of Serre's results \cite[Prop 10]{Se55} and \cite[Thm 1]{Se55}.
\end{Rem}
%------------------------------------------------------------------------------------------
The following result is due to Deligne, see \cite[Thm 5.5]{De68}, for smooth morphisms $f:\sY\to S$. His proof also works in our situation. As his arguments are part of the proof of a more general statement, we reproduce them here.
%------------------------------------------------------------------------------------------
\begin{Thm}[\textbf{Deligne}]\label{thm deligne}
Let $Y$ be a proper, simple normal crossing $\C$-variety, let $f:\sY \to S=\Spec R$  for $R\in\Art_\C$ be a locally trivial deformation of $Y$ over $S$ and let $S' \to S$ be a morphism, where $S'=\Spec R'$ for $R' \in \Art_\C$. Then the following holds.
\begin{enumerate}
	\item \label{spec seq} The associated spectral sequence
	\begin{equation}\label{Kf spec seq}
	E_1^{p,q}=R^qf_*\tOm_{\sY/S}^p \Rightarrow R^{p+q}f_*\tOm_{\sY/S}^\bullet = H^{p+q}(Y^\an,\ul{R}_{Y^\an})
	\end{equation}
	degenerates at $E_1$.
	\item \label{lemma deligne free} The $R$-modules $R^qf_*\tOm_{\sY/S}^p$ are free and compatible with arbitrary base change in the sense that for $\sY'=\sY\times_{S}S'$ the morphism $$R^qf_*\tOm_{\sY/S}^p\otimes_R R' \to R^qf_*\tOm_{\sY'/S'}^p$$ is an isomorphism.
\end{enumerate}
The analogous statements hold if $f:\sY \to S$ is 
replaced by a deformation $\sX \to S$ of a compact K\"ahler manifold $X$.
\end{Thm}
\begin{proof} We argue as in \cite{De68}, Th\'eor\`eme 5.5 for the morphism $f:\sY\to S$. By \cite[(3.5.1)]{De68} a complex $K$ of $R$-modules satisfies
\[
\lg_R(H^n(K)) \leq \lg(R) \dim_\C (H^n(K\otimes_R^\L \C))
\]
and $H^n(K)$ is a free $R$-module if equality holds. Here $\lg$ denotes the length of a module. To apply this to the $E_1$-term of the spectral sequence \eqref{Kf spec seq} we need \cite{EGAIII2}, Th\'eor\`eme (6.10.5) saying that there is a bounded below complex $L$ of free $R$-modules and an isomorphism of $\partial$-functors $R^q f_*\left(\tOm_{\sY/S}^p \otimes f^*Q\right) \to H^q(L\otimes Q)$ in the bounded complex $Q$ of quasi-coherent $R$-modules. Here we use that $\tOm_{\sY/S}^\bullet$ is flat over $R$. Let $\bar{f}:Y \to \Spec\C$ be the restriction of $f$ to the central fiber. We will compare the spectral sequence \eqref{Kf spec seq} with the spectral sequence of $\bar{f}$. Again by \cite[(3.5.1)]{De68} we have 
\begin{equation}\begin{aligned}\label{second ineq}
\lg_R(R^qf_*\tOm_{\sY/S}^p)  &= \lg_R(H^q(L)) \\ &\leq \lg(R) \dim_\C (H^q(L\otimes_R \C))\\ &= \lg(R) \dim_\C (R^q\bar{f}_*\tOm_{Y/\C}^p)
\end{aligned}\end{equation}
and $R^qf_*\tOm_{\sY/S}^p$ is a free $R$-module if equality holds. 
We have
\begin{align*}
\lg (R^nf_*\tOm_{\sY/S}^\bullet) &\leq \sum_{p+q=n} \lg_R(R^qf_*\tOm_{\sY/S}^p)\\
& \leq \lg(R) \sum_{p+q=n}\dim_\C (R^q\bar{f}_*\tOm_{Y/\C}^p)\\
&= \lg(R)\dim_\C (R^n\bar{f}_*\tOm_{Y/\C}^\bullet),
\end{align*}
where the first inequality comes from the existence of the spectral sequence, the second inequality is \eqref{second ineq} and the last equality comes from the degeneration of the spectral sequence for $Y$, which is \cite[Prop 1.5]{Fr83}. But Lemma \ref{lemma friedman} \eqref{zwei} implies that $\lg (R^nf_*\tOm_{\sY/S}^\bullet) = \lg(R)\dim_\C (R^n\bar{f}_*\tOm_{Y/\C}^\bullet)$, so we have equality everywhere. Hence \eqref{spec seq} and the first assertion of \eqref{lemma deligne free} follows. The second assertion of \eqref{lemma deligne free} follows from the first by \cite[(7.8.5)]{EGAIII2}.

The K\"ahler case works literally as above, we only have to replace the reference to \cite[Thm 6.10.5]{EGAIII2} by \cite[Ch~3, Thm 4.1]{BS} and the reference to \cite[7.8.5]{EGAIII2} by \cite[Ch~3, Cor 3.10]{BS}. The rest of the proof of Theorem \ref{thm deligne} goes through if we note that the spectral sequence associated with $\Omega_X^\bullet$ degenerates as $X$ is a compact K\"ahler manifold.
\end{proof}
%------------------------------------------------------------------------------------------
\subsection{Pure Hodge structures on smooth families}\label{subsec hodge structure smooth}
%------------------------------------------------------------------------------------------
Let $f:\sY\to S$ be a smooth and proper morphism of complex spaces where $S=\Spec R$ for $R\in \Art_\C$. We are going to put a pure Hodge structure over $R$ on $H^k(Y,\ul{R}_{Y})$ where $Y=\sY^\red$. The decreasing filtration $F^p\Omega_{\sY/S}^\bullet:= \Omega_{\sY/S}^{\geq p}$ gives rise to the \emph{Hodge filtration} $F^pH^k(Y,\ul{R}_{Y})$ on $H^k(Y,\ul{R}_{Y})$, which we obtain by setting
\begin{equation}\label{hodge filtration pure}
F^p R^kf_*\Omega_{\sY/S}^\bullet := \img\left(R^k f_*F^p\Omega_{\sY/S}^\bullet\to R^k f_*\Omega_{\sY/S}^\bullet\right)
\end{equation}
and using the isomorphisms $H^k(Y,\ul{R}_{Y}) \to R^k f_*\Omega_{\sY/S}^\bullet$ from \cite[Lem 5.5.3]{De68}.
%------------------------------------------------------------------------------------------
\begin{Lem}\label{lemma pure}
Let $f:\sY \to S=\Spec R$ be a smooth and proper morphism of complex spaces where $R\in \Art_\C$. Then 
\[
\sH^k(\sY):=\left(H^k(Y,\R), F^p H^k(Y,\ul{R}_{Y})\right)
\]
is a pure Hodge structure of weight $k$ over $R$, whose central fiber is the usual Hodge structure on $H^k(Y, \R)$. Moreover, the canonical morphism $R^k f_*F^p\Omega_{\sY/S}^\bullet\to R^k f_*\Omega_{\sY/S}^\bullet$ is injective, so that $R^k f_*F^p\Omega_{\sY/S}^\bullet \isom F^p H^k(Y,\ul{R}_{Y}).$ If $g:\sX\to S$ is smooth and proper, every $S$-morphism $i:\sY\to \sX$ induces a morphism $i^*:\sH^k(\sX)\to \sH^k(\sY)$ of pure Hodge structures over $R$.
\end{Lem}
\begin{proof} The filtration defined in \eqref{hodge filtration pure} is the one, whose graded objects are found on $E_\infty$ of the spectral sequence \eqref{Kf spec seq}. By \cite[Thm 5.5]{De68} we have $E_\infty=E_1$, so $\Gr_F^p R^k f_*\Omega_{\sY/S}^\bullet = R^{k-p}f_*\Omega_{\sY/S}^p = R^{k-p}f_* \Gr_F^p \Omega_{\sY/S}^\bullet$. The same theorem tells us that $R^{k-p}f_*\Omega_{\sY/S}^p $ is free. Therefore using
\[
\begin{xy}
\xymatrix{
0 \ar[r] & R^kf_*F^{p+1} \Omega_{\sY/S}^\bullet \ar[r] \ar[d] & R^kf_*F^p \Omega_{\sY/S}^\bullet \ar[r] \ar[d] & R^kf_*\Gr_F^p \Omega_{\sY/S}^\bullet \ar[d] \ar[r] & 0 \\
0 \ar[r] & F^{p+1} R^kf_*\Omega_{\sY/S}^\bullet \ar[r] & F^{p} R^kf_*\Omega_{\sY/S}^\bullet \ar[r] & \Gr_F^p R^kf_*\Omega_{\sY/S}^\bullet \ar[r] & 0
}
\end{xy}
\]
we find inductively that $R^kf_*F^p \Omega_{\sY/S}^\bullet \isom F^p R^kf_*\Omega_{\sY/S}^\bullet$ and that these are free submodules. Again by \cite[Thm 5.5]{De68}, all graded objects are compatible with base change and therefore restrict to a pure Hodge structure on the central fiber. The statement about morphisms is clear.
\end{proof}
%------------------------------------------------------------------------------------------
\begin{Cor}\label{cor pure grf}
There is a natural isomorphism $$R^{k-p} f_*\Omega_{\sY/S}^p \to \Gr_F^p H^k(Y,R).$$ 
\end{Cor}
\begin{proof} 
Consider the sequences 
\begin{equation}\label{grf iso}
\begin{xy}
\xymatrix{
0 \ar[r] & R^kf_* \Omega_{\sY/S}^{\geq p+1} \ar[r] \ar[d]^\isom & R^kf_* \Omega_{\sY/S}^{\geq p} \ar[r] \ar[d]^\isom & R^{k-p}f_* \Omega_{\sY/S}^{p} \ar[r] \ar@{-->}[d]^{\exists\,?} & 0\\
0 \ar[r] & F^{p+1}H^k(Y,R) \ar[r] & F^{p}H^k(Y,R) \ar[r] & \Gr_F^{p}H^k(Y,R) \ar[r] & 0
}
\end{xy}
\end{equation}
where the first two vertical maps are isomorphisms by Lemma \ref{lemma pure}. These isomorphisms imply that the upper sequence is exact on the left. As it is part of the long exact sequence associated with the sequence
\[
0\to \Omega_{\sY/S}^{\geq p+1} \to \Omega_{\sY/S}^{\geq p} \to \Omega_{\sY/S}^p[-p] \to 0
\]
of complexes, injectivity at the $(k+1)$-st direct image yields surjectivity at the $k$-th, hence exactness of the upper sequence. Therefore, the morphism $R^{k-p} f_*\Omega_{\sY/S}^p \to \Gr_F^p H^k(Y,R)$ exists and by the five-lemma it is an isomorphism.
\end{proof}
%------------------------------------------------------------------------------------------
\begin{Cor}\label{cor pure hpq}
There is a natural isomorphism  
\[
 \left(R^{k-p} f_*\Omega_{\sY/S}^p\right)_\wls \to[\isom]H^{p,q}(Y,R) := F^p_\wls \cap \ba{F^q_\wls} \subset H^k(Y,\R)\tensor R,
\]
which is functorial in $\sY$.
\end{Cor}
\begin{proof} This is deduced directly by applying Weil restriction to the diagram \eqref{grf iso} and using Lemma \ref{lemma hodge decomposition}.
\end{proof}
%------------------------------------------------------------------------------------------
Recall that a module homomorphism has constant rank if and only if its cokernel is free by Lemma \ref{lemma constant rank}.
%------------------------------------------------------------------------------------------
\begin{Prop}\label{proposition free smooth}
Let $f:\sY \to S$, $g:\sX\to S$  be proper and smooth over $S=\Spec R$, $R\in \Art_\C$ and let $i:\sY \to \sX$ be an $S$-morphism. Then the induced morphisms $i^* :R^qg_*\Omega^p_{\sX/S} \to R^qf_*\Omega^p_{\sY/S}$ have constant rank.
\end{Prop}
\begin{proof}
By Lemma \ref{lemma pure} we know that the morphism $i$ induces a morphism $\sH^k(\sX)~\to~\sH^k(\sY)$ between the pure Hodge structures over $R$ associated with $\sX$ and $\sY$. Taking Weil restrictions this gives a morphism $\sH^k(\sX)_\wls \to \sH^k(\sY)_\wls$ of Hodge-Weil structures by Lemma \ref{lemma weil hodge = hodge weil}. Let $i^{p,q}:H^{p,q}(X) \to H^{p,q}(Y)$ be the induced map.
By Corollary \ref{cor pure hpq} the diagram
\begin{equation*}\label{coker diagram}
\begin{xy}
\xymatrix{
\left(R^{k-p} f_*\Omega_{\sY/S}^p\right)_\wls \ar[r]^{i^*_\wls} \ar[d]^\isom & \left(R^{k-p} g_*\Omega_{\sX/S}^p\right)_\wls \ar[d]^\isom \ar[r] & \coker i^*_\wls \ar[d] \ar[r] &0\\
H^{p,q}(Y)   \ar[r]^{i^{p,q}} & H^{p,q}(X) \ar[r]& \coker{i^{p,q}} \ar[r] &0\\
}
\end{xy}
\end{equation*}
commutes and the first two vertical maps are isomorphisms. Therefore, also the third vertical map is an isomorphism. We know that $\coker i^{p,q}$ is free by Lemma \ref{lemma morphisms}, hence so is $\coker i^*_\wls$. Now the claim follows from Lemma \ref{lemma coker}, as $\coker i^*_\wls = \left(\coker i^*\right)_\wls$ by Lemma \ref{lem weil restriction}.
\end{proof}
%------------------------------------------------------------------------------------------
Proposition \ref{proposition free smooth} together with Lemma \ref{lemma morphisms} can be seen as a formalization of the following argument: If $S$ is the base manifold of a small deformation and $t\in S$, the maps $H^q(X_t,\Omega_{X_t}^p) \to H^q(Y_t,\Omega_{Y_t}^p)$, the rank of which is semi-continuous in $t$, add up to the topological map $H^i(X_t,\C) \to H^i(Y_t,\C)$ by the Hodge decomposition. The rank of the latter is independent of $t$ and by semi-continuity the summands also have constant rank. 
%------------------------------------------------------------------------------------------
\subsection{Mixed Hodge structures on normal crossing families}\label{subsec hodge structure nc}
%------------------------------------------------------------------------------------------
Let $Y$ be a simple normal crossing $\C$-variety and $f:\sY \to S$ be a locally trivial deformation of $Y$ over $S=\Spec R$ with $R \in \Art_\C$. Here we need for the first time that $\sY$ is a scheme rather than a complex space, because we want to invoke Lemma \ref{lemma irreducibles}. However, this need is probably only due the approach we chose. The analogue of Lemma \ref{lemma irreducibles} should be valid for complex spaces, too.

By Lemma \ref{lemma friedman} \eqref{zwei} there is a quasi-isomorphism $\tOm_{\sY/S}^\bullet \simeq \goths({(a_\bullet)}_*\Omega_{\sY^{\bullet}/S}^\bullet)$, where $\goths(\cdot)$ denotes the single complex associated with a double complex. 
We define filtrations 
$W_{-m} \tOm_{\sY/S}^\bullet := \goths(({a_{\geq m}})_*\Omega_{\sY^{\geq m}/S}^\bullet)$ 
and 
$F^p\tOm_{\sY/S}^\bullet:= \tOm_{\sY/S}^{\geq p}$.
These give rise to filtrations $F^pH^k(Y,R)$ and $W_mH^k(Y,R)$ on $H^k(Y,R)$ if we put
\begin{equation}\label{weight filtration mixed}
W_m R^kf_*\tOm_{\sY/S}^\bullet := \img\left(R^kf_*W_{m-k}\tOm_{\sY/S}^\bullet\to R^kf_*\tOm_{\sY/S}^\bullet\right)
\end{equation}
and
\begin{equation}\label{hodge filtration mixed}
F^p R^kf_*\tOm_{\sY/S}^\bullet := \img\left(R^kf_*F^p\tOm_{\sY/S}^\bullet\to R^kf_*\tOm_{\sY/S}^\bullet\right)
\end{equation}
and use the isomorphisms $H^k(Y^\an,\ul{R}_{Y^\an}) \to R^k f^\an_*\tOm_{\sY^\an/S}^\bullet$ from Lemma \ref{lemma friedman} \eqref{zwei} and $R^k f_*\tOm_{\sY/S}^\bullet \to R^k f^\an_*\tOm_{\sY^\an/S}^\bullet$ from Lemma~\ref{lemma friedman} \eqref{vier}.
%------------------------------------------------------------------------------------------
\begin{Thm}\label{thm mixed}
Let $Y$ be a proper simple normal crossing variety over $\C$ and let $f:\sY \to S$ be a locally trivial deformation of $Y$ over $S=\Spec R$ for $R\in \Art_\C$. Then 
\begin{equation}\label{is mhs over r}
\sH^k(\sY)=(H^k(Y^\an,\R), W_m H^k(Y^\an,\ul{R}_{Y^\an}), F^p H^k(Y^\an,\ul{R}_{Y^\an}))
\end{equation}
is a mixed Hodge structure over $R$. Moreover, $R^k f_*F^p\tOm_{\sY/S}^\bullet\to R^k f_*\tOm_{\sY/S}^\bullet$ is injective so that $F^p H^k(Y^\an,\ul{R}_{Y^\an}) \isom R^k f_*\tOm_{\sY/S}^{\geq p}$.
\end{Thm}
%------------------------------------------------------------------------------------------
\begin{proof} Literally as in the pure case, see Lemma \ref{lemma pure} and Corollary \ref{cor pure grf}, one shows that the $R$-modules $\Gr_F^p R^kf_*\tOm_{\sY/S}^\bullet$ are free and isomorphic to $R^kf_*\Gr_F^p\tOm_{\sY/S}^\bullet=R^kf_*\tOm_{\sY/S}^p$ and that $R^k f_*F^p\tOm_{\sY/S}^\bullet\to R^k f_*\tOm_{\sY/S}^\bullet$ is injective. The only difference is that one has to use Theorem \ref{thm deligne} instead of \cite[Thm 5.5]{De68}. To verify that \eqref{is mhs over r} is a mixed Hodge structure over $R$, we have to show that the graded objects $\Gr_m^W\Gr_F^p H^k(Y^\an,\ul{R}_{Y^\an})$ are free $R$-modules, or equivalently that the $\Gr_m^W\Gr_F^p R^kf_*\tOm_{\sY/S}^\bullet$ are free $R$-modules, and that the central fiber is a mixed Hodge structure in the ordinary sense. 

The free $R$-module $R^kf_*\tOm_{\sY/S}^p$ is the abutment of the spectral sequence
\begin{equation}\label{weight spectral sequence}
E_1^{k, m}=R^m f_*\;{a_k}_*\Omega_{\sY^{k}/S}^p \Rightarrow R^{k+m}f_*\tOm_{\sY/S}^p
\end{equation}
induced by the resolution \eqref{resolution} for fixed $p$. The filtration defined in \eqref{weight filtration mixed} induces a weight filtration $\Gr_F^p R^kf_*\tOm_{\sY/S}^\bullet$ in the obvious way and the graded objects with respect to this filtration are the $E_\infty$ terms of the spectral sequence \eqref{weight spectral sequence}. By \cite[Thm 5.5]{De68} the $R$-modules $E_1^{k, m}$ are free and compatible with base change. Moreover, the differential $d_1$ on $E_1^{k, m}$ is given by the semi-simplicial differential
$\delta: R^m f_*{a_k}_*\Omega_{\sY^{k}/S}^p \to R^m f_*{a_k}_*\Omega_{\sY^{k+1}/S}^p.$

This morphism  has constant rank by Proposition \ref{proposition free smooth}. Hence $E_2^{k, m}$ is free, too, and compatible with base change by Lemma \ref{lem const rank cohomology}. In the case $R=\C$ the spectral sequence is known to degenerate at $E_2$, see \cite[Thms 3.12, 3.18]{PS}. As all $E_2$-terms of \eqref{weight spectral sequence} are compatible with base change we have for all $n$ that
\begin{align*}
\sum_{k + m = n} \lg_R \left(E_2^{k,m}\right)  &= \lg_R(R) \sum_{k + m = n} \dim_\C \left(E_2^{k,m}\tensor\C\right)  \\
&= \lg_R(R) \dim_\C\left(R^n f_*\tOm_{Y/\C}^p\right)\\
 &= \lg_R \left(R^n f_*\tOm_{\sY/S}^p\right).
\end{align*}
Thus, the spectral sequence \ref{weight spectral sequence} also degenerates at $E_2$ and the $R$-modules $E_\infty^{k,m} = \Gr_m^W R^{k+m} f_*\tOm_{\sY/S}^p = \Gr_m^W \Gr_F^p R^{k+m} f_*\tOm_{\sY/S}^\bullet$ coincide with the free $R$-modules $E_2^{k,m}$. Again, as all graded objects are compatible with base change, $\sH$ restricts to a mixed Hodge structure on the central fiber, which is the usual mixed Hodge structure on $Y$.
\end{proof}
%------------------------------------------------------------------------------------------
Let us isolate an observation from the proof of the previous lemma.
\begin{Cor}\label{cor weight seq}
Let $Y$ be a proper simple normal crossing variety over $\C$ and let $f:\sY \to S$ be a locally trivial deformation of $Y$ over $S=\Spec R$ for $R\in \Art_\C$. Then the spectral sequence \eqref{weight spectral sequence} degenerates at $E_2$.\qed
\end{Cor}
%------------------------------------------------------------------------------------------
\begin{Thm}\label{thm free singular}
Let $S=\Spec R$ where $R\in\Art_\C$, let $Y$ be a proper simple normal crossing $\C$-variety and let $g:\sX\to S$ and $f:\sY\to S$ be proper, algebraic $S$-schemes. Assume that $\sY \to S$ is a locally trivial deformation of $Y$ and that $\sX\to S$ is smooth. Let $i:\sY \to \sX$ be an $S$-morphism. Then for all $p,q$ the morphism $i^*:R^qg_*\Omega^p_{\sX/S} \to R^qf_*\tOm^p_{\sY/S}$ has constant rank.
\end{Thm}	
\begin{proof}
Let $\begin{xy}
\xymatrix@C=1em{
\ldots  \ar[r]\ar@<0.5ex>[r]\ar@<-0.5ex>[r] & \sY^{1} \ar@<0.5ex>[r]\ar@<-0.5ex>[r] & \sY^{0} \ar[r]& \sY \\
}
\end{xy}$ be the semi-simplicial resolution of $\sY$ over $S$ from Lemma \ref{lemma can res}. This means in particular that $\sY^{0}$ is a locally trivial deformation of the normalization. By Theorem \ref{thm deligne} the $R$-modules $R^qg_*\Omega^p_{\sX/S}$, $R^qf_*\tOm^p_{\sY/S}$ and $R^qf_*\Omega^p_{\sY^{k}/S}$ are free and compatible with base change. By Corollary \ref{cor weight seq} we know that the spectral sequences \eqref{weight spectral sequence} degenerate at $E_2$ for each $p$. As
$E_2^{0,q} = \ker\left(R^q f_*\Omega^p_{\sY^{0}/S} \to R^q f_*\Omega^p_{\sY^{1}/S}\right),$
this implies that the first row in
\begin{equation}\label{smooth to singular morphism}
\begin{xy}
\xymatrix@C=1.5em{
0\ar[r] &W_{p+q-1} R^q f_*\tOm^p_{\sY/S} \ar[r] & R^q f_*\tOm^p_{\sY/S} \ar[r]^(0.45)\eta & R^q f_*\Omega^p_{\sY^{0}/S} \ar[r]^\delta & R^q f_*\Omega^p_{\sY^{1}/S}\\
&& R^q g_*\Omega^p_{\sX/S} \ar[u]^{i^*} \ar[ru]_\vphi & &\\
}
\end{xy}
\end{equation}
is exact. 

Here $\img i^*$ does not intersect $W_{p+q-1} R^q f_*\tOm^p_{\sY/S}$, as it does not on the central fiber. This last claim can be shown using Deligne's weak splitting as follows. We denote $X:=\sX\times_S\C$ and put $H_Y^{p+q}:=H^{p+q}(Y,\C)$ and $H_X^{p+q}:=H^{p+q}(X,\C)$. We identify $H_Y^{p+q}$ and $H_X^{p+q}$ with the hypercohomology of $\tOm_Y^\bullet$ respectively $\Omega_X^\bullet$ and obtain
\begin{equation*}
\begin{xy}
\xymatrix@R=1em{
\left(W_{p+q-1} R^q f_*\tOm^p_{\sY/S}\right) \tensor \C \ar@{^(->}[r] \ar@{=}[d] & \left(R^q f_*\tOm^p_{\sY/S}\right) \tensor \C \ar@{=}[d] & \left(R^q g_*\Omega^p_{\sX/S}\right) \tensor \C \ar[l]_{i^*} \ar@{=}[d] \\
W_{p+q-1} R^q f_*\tOm^p_{Y} \ar@{^(->}[r] \ar@{=}[d] & R^q f_*\tOm^p_{Y} \ar@{=}[d] & R^q g_*\Omega^p_{X} \ar[l] \ar@{=}[d] \\
W_{p+q-1} \Gr^p_F H_Y^{p+q} \ar@{^(->}[r] & \Gr^p_F H_Y^{p+q} & \Gr^p_F H_X^{p+q} \ar[l]
}
\end{xy}
\end{equation*}
Deligne's weak splitting \cite[Ex 3.3 and Lem-Def 3.4]{PS} is a decomposition $H_Y^k = \bigoplus_{r,s} \, I_Y^{r,s}$
such that
\[
F^pH_Y^k = \bigoplus_{r\geq p} \ I_Y^{r,s} \quad {\rm and} \quad W_m H_Y^k = \bigoplus_{r+s\leq m} I_Y^{r,s}.
\]
The subspaces $I_Y^{r,s}\subset H^{r+s}_Y$ project isomorphically onto the subquotients $\Gr_{r+s}^W\Gr_F^r H_Y$. The Deligne weak splitting is preserved under morphisms of mixed Hodge structures. As the Hodge structure on $H_X^{p+q}$ is pure of weight $p+q$, this yields $\img i^*\subset I^\pq_Y$ and therefore 
\[
\img i^* \cap W_{p+q-1} \Gr^p_F H_Y^{p+q} \; \subset \; I^\pq_Y \cap \bigoplus_{r+s\leq p+q-1} I_Y^{r,s} = 0 .
\]
as claimed. 

We come back to diagram \eqref{smooth to singular morphism} and observe that $\vphi$ has constant rank by Proposition \ref{proposition free smooth}. Also $\eta$ has constant rank as $\delta$ has constant rank by Proposition \ref{proposition free smooth} and hence $\coker \eta = \ker \delta$ is free.  As $\img i^*\cap \ker\eta =0$, Lemma \ref{lem const rank dreieck} implies that $i^*$ has constant rank completing the proof.
\end{proof}
%------------------------------------------------------------------------------------------
\subsection{Vista} It seems obvious that the results of this article are only a master example of what is true in general. We belief that the statements of Theorem \ref{thm mixed}  and Theorem \ref{thm free singular} should hold true mutatis mutandis for all locally trivial varieties $\sX\to S$ and $\sY\to S$ over an Artinian base scheme. One way to establish these results would be to prove resolution of singularities in this setting and then produce semi-simplicial resolutions as in Deligne's approach. On the other hand it is clear that the local triviality condition is necessary.
%------------------------------------------------------------------------------------------

\appendix 

%------------------------------------------------------------------------------------------
\section{Commutative algebra}\label{sec commalg}
%------------------------------------------------------------------------------------------
Let $R$ be a noetherian ring and $\vphi:F\to G$ be a morphism between finitely generated free $R$-modules. We define $I_j(\vphi)=\img(\vphi':\Lambda^j F\tensor (\Lambda^j G)^\vee \to R)$, where $\vphi'$ is induced by $\Lambda^j\vphi:\Lambda^jF\to\Lambda^jG$. If we interpret $\vphi$ as a matrix, then $I_j(\vphi)$ is the ideal generated by all $j\times j$-minors of $\vphi$. If $F$ and $G$ are finitely generated but not necessarily free, the definition still makes sense if $G$ is projective. One defines the \emph{rank of $\vphi$} as $\rk\vphi:= \max \left\{i:I_i(\vphi)\neq 0\right\}$. 
%------------------------------------------------------------------------------------------
\begin{Def}\label{definition constant rank}
Let $R$ be a noetherian ring and $\vphi:F\to G$ be a morphism between finitely generated $R$-modules. Suppose $G$ is projective. We say that $\vphi$ has \emph{constant rank $k$} if $I_{k}(\vphi)=R$ and $I_{k - 1}(\vphi)=0$. We say that $\vphi$ has \emph{constant rank} if there is some $k$ such that $\vphi$ has constant rank $k$.
\end{Def}
%------------------------------------------------------------------------------------------
A characterization of this property is given by the following Lemma, the proof of which is found at \cite[Prop 20.8]{E}.
%------------------------------------------------------------------------------------------
\begin{Lem}\label{lemma constant rank}
Let $R$ be a noetherian ring and $\vphi:F\to G$ be a morphism between finitely generated $R$-modules. Suppose $G$ is projective. Then $\vphi$ has constant rank if and only if $\coker\vphi$ is a projective $R$-module.\qed
\end{Lem}
%------------------------------------------------------------------------------------------
\begin{Lem}\label{lem free qoutient module}
Let $(R,\gothm)$ be a local Artin ring with residue field $k$ and $F_1 \subset F$ two finitely generated free $R$-modules. Then $F/F_1$ is free and $\vphi:F_1\tensor k \to F \tensor k$ is injective.
\end{Lem}
\begin{proof}
If $F_1\to F$ is injective, then $F/F_1$ is free if and only if $\vphi:F_1\tensor k \to F \tensor k$ is injective. This holds over any local noetherian ring by \cite[Cor A.6]{Se}. As both $F_1$ and $F$ are free, the diagram
\[
\begin{xy}
\xymatrix{
0 \ar[r] & \gothm F_1 \ar[r] \ar@{^(->}[d]& F_1 \ar[r] \ar@{^(->}[d]& F_1\tensor k \ar[r] \ar[d]^\vphi& 0\\
0 \ar[r] & \gothm F \ar[r] & F \ar[r] & F\tensor k \ar[r] & 0\\
}
\end{xy}
\]
has exact rows. If $\vphi$ is not injective, then there is $x_1\in F_1 \cap \gothm F$ with $x_1 \notin \gothm F_1$. Because of this last property we find $x_2,\ldots, x_k \in F$ such that $x_1,x_2,\ldots, x_k$ is a basis of $F$ by Nakayama's Lemma. In particular, if $\alpha x_1 = 0$ for some $\alpha \in R$, then $\alpha=0$. The case $R=k$ is trivial, so we may assume that the maximal ideal $\gothm$ is non-zero. So there is $0\neq\alpha\in\Ann\gothm$. Therefore we have $\alpha x_1 \in \alpha \gothm F =0$, a contradiction.
\end{proof}
%------------------------------------------------------------------------------------------
\begin{Cor}\label{cor intersection of free}
Let $(R,\gothm)$ be a local Artin ring with residue field $k$ and $F_1, F_2 \subset F$ be two free submodules in a finitely generated free $R$-module. Then $F_1\cap F_2 =0$ if and only if $F_1\tensor k \cap F_2 \tensor k = 0$.
\end{Cor}
\begin{proof} The condition $F_1\cap F_2 =0$ means that $F_1\oplus F_2 \to F$ is injective. This implies injectivity of $F_1\tensor k\oplus F_2 \tensor k\to F\tensor k$  by Lemma \ref{lem free qoutient module}, hence $F_1\tensor k\cap F_2 \tensor k=0$. The converse again follows from \cite[Cor A.6]{Se} over any local noetherian ring.
\end{proof}
%------------------------------------------------------------------------------------------
\begin{Lem}\label{lem const rank dreieck}
Let $R$ be a local Artin ring and 
\[
\begin{xy}
\xymatrix{
F \ar[dr]^\vphi \ar[d]_\psi&\\
G \ar[r]^\eta &H\\
}
\end{xy}
\]
a diagram of $R$-modules where $G,H$ are free, $\eta$ has constant rank and $\img \psi \cap \ker \eta = 0$. Then $\vphi$ has constant rank if and only if $\psi$ has.
\end{Lem}
\begin{proof}
We may assume that $\psi$ is injective since replacing $F$ by $\img\psi$ does not change any cokernel. As $\eta$ has constant rank, $\ker\eta$ and $\coker\eta=H/\eta(G)$ are free. If we consider the two exact sequences
\[
0\to G/(F\oplus\ker\eta) \to H/\vphi(F) \to H/\eta(G) \to 0
\]
and
\[
0\to \ker\eta \to G/F \to G/(F\oplus\ker\eta) \to 0
\]
we see that $H/\vphi(F)$ is free if and only if $G/(F\oplus\ker\eta)$ is free. By Lemma \ref{lem free qoutient module} this is the case if and only if $G/F$ is free. This proves the Lemma.
\end{proof}
%------------------------------------------------------------------------------------------
\begin{Lem}\label{lem const rank cohomology}
Let $R$ be a local Artin ring and 
\[
H' \stackrel{d_1}{\rightarrow} H \stackrel{d_2}{\rightarrow} H''
\]
a complex of free $R$-modules, i.e. $d_2\circ d_1=0$. If the $d_i$ have constant rank, then the cohomology $\ker d_2/\img d_1$ is free.
\end{Lem}
\begin{proof}
Consider the diagram
\[
\begin{xy}
\xymatrix{
H' \ar@{->>}[d] \ar[r]^{d_1} & H \ar[r]^{d_2} & H'' \\
F \ar@{^(->}[ru]^\vphi \ar[r]^\psi & G \ar@{^(->}[u]^\eta &\\
}
\end{xy}
\]
where $F=\img d_1$ and $G=\ker d_2$. Here $F$ and $G$ are free by the remarks following Definition \ref{definition constant rank}, hence the claim follows from Lemma \ref{lem free qoutient module}.
\end{proof}
%------------------------------------------------------------------------------------------

\end{document}